\documentclass[twoside]{amsart}
\usepackage{amssymb,amsfonts}
\usepackage{amsthm}
\usepackage{amsmath}
\usepackage{graphicx}

\newcommand{\beq}{\begin{equation}}
\newcommand{\eeq}{\end{equation}}
\newcommand{\beqs}{\begin{equation*}}
\newcommand{\eeqs}{\end{equation*}}
\newcommand{\ba}{\begin{array}}
\newcommand{\ea}{\end{array}}
\newcommand{\beas}{\begin{eqnarray*}}
\newcommand{\eeas}{\end{eqnarray*}}
\newcommand{\bea}{\begin{eqnarray}}
\newcommand{\eea}{\end{eqnarray}}
\newcommand{\bal}{\begin{align}}
\newcommand{\eal}{\end{align}}

\newcommand{\bals}{\begin{align*}}
\newcommand{\eals}{\end{align*}}

\newcommand{\R}{\ensuremath{\mathbb R}}

\newcommand{\bds}{\begin{displaystyle}}
\newcommand{\eds}{\end{displaystyle}}

\renewcommand{\eqref}[1]{(\ref{#1})}

\def\longequals{\mathbin{=\kern-2pt=}}

\newcommand{\remove}[1]{} 
\renewcommand{\remove}[1]{#1} 

\newtheorem{theorem}{Theorem}[section]

\newtheorem{lemma}[theorem]{Lemma}

\newtheorem{proposition}[theorem]{Proposition}
\newtheorem{definition}[theorem]{Definition}

\newtheorem{remark}[theorem]{\bf{Remark}}
\newtheorem{example}[theorem]{\bf{Example}}

\title{Geometric Methods in the Analysis of Non-linear Flows in Porous Media}
\author{Eugenio Aulisa, Akif Ibragimov and Magdalena Toda}
\address{Department of Mathematics and Statistics, Texas Tech University, Box 41042, Lubbock, TX 79409--1042, U. S. A.}
\email{eugenio.aulisa@ttu.edu}
\email{akif.ibraguimov@ttu.edu}
\email{magda.toda@ttu.edu}

\begin{document}
\maketitle

\begin{abstract}
Over the past few years, we developed a mathematically
rigorous method to study the dynamical processes associated to nonlinear Forchheimer flows for slightly compressible fluids.
We have proved the existence of a geometric transformation which relates constant mean curvature surfaces and time-invariant pressure distribution graphs constrained by the Darcy-Forchheimer law. We therein established a direct relationship between the CMC graph equation and a certain family of equations which we call $g$-Forchheimer equations. The corresponding results, on fast flows and their geometric interpretation, can be used as analytical tools in evaluating important technological parameters in reservoir engineering.
\end{abstract}
%

\section{Background}
The present report uses methods of differential geometry and integrable systems in
 modeling nonlinear flows in porous media. The classical
flow, as defined by Darcy's equation \cite{Forch,Muskat,Bear,Dake}
represents a linear relation between the gradient of pressure and velocity
field. Darcy's formulation of the momentum equation of motion essentially
simplifies some complicated hydrodynamics in porous media, by using a
permeability tensor as the major characteristic of the media at a point (pore, block, layer, etc)
in space.

Most of the up-scaling techniques in heterogeneous reservoirs  are based on the
assumption that on each scaling level, the constitutive momentum relation
between the gradient of pressure and velocity field remains invariant, while
the only parameter that is reevaluated is the permeability tensor.

Philipp Forchheimer, one of the founders of groundwater hydrology at the beginning of the 20th century,
observed that for high velocity rates, Darcy's law is no longer valid. On the other hand,
on a fine level, it is well-known that even slow flows can deviate from the linear ones \cite{Muskat}.

Some newer theories suggest that it is precisely the inertial term that
causes the deviation from the linear Darcy law, but the nature of
these inertial forces is not fully understood. As
it was pointed by Bear, \cite{Bear} (see also \cite{Whitaker,Tavera,Miskimins}): ``Most of the experiments indicate
that the actual turbulence occurs at a Reynolds number at least one order of
magnitude higher than the Reynolds value at which the deviation from
Darcy's law is observed''. In some of the experiments, a deviation from
Darcy's law was observed for Reynolds number $Re=qd/\mu \simeq 10$ (where $d$
is some length dimension of the porous matrix, $q$ a specific discharge per
unit cross-sectional area normal to the direction of the flow, and $\mu $ is
viscosity of the liquid). In some recent works, such as
\cite{Tavera}, it was
experimentally observed that Darcy's law is not verified even for $Re\simeq
1$, for samples of the porous rocks containing fracture. Latest results
suggest that even some low velocity flows in a highly heterogeneous
reservoir may deviate from Darcy's linearity principle.

In reservoir engineering, there are three most popular non-linear
approximations of the field data, establishing a formula for the pressure drop $\Delta P$ in terms of the production rate $Q$:
\begin{itemize}
\item[-] the ``two term'' law - $AQ+BQ^2=\Delta P$;
\item[-] the ``power'' law - $CQ^n+aQ=\Delta P,\;\;\;\;\;1\leqslant n\leqslant 2$;
\item[-] the ``three term'' cubic law - $AQ+BQ^2+CQ^3=\Delta P$.
\end{itemize}
All three equations were originally introduced by Forchheimer in works published at the beginning of the 20th 
century.

Following \cite {Whitaker,SIAM, MF, AIVW3} and
references herein we have established three types of generalized nonlinear
Darcy law, with permeability tensor depending on the gradient of the pressure
function. We will show that these constitutive equations correspond to a generalized
Forchheimer equation. Under some assumption about fluids, the
generalized Darcy-Forchheimer equations enable a reduction of the system of equations
that governs the flow - namely to one parabolic non-linear equation for the
pressure function. This parabolic equation displays some similarity to the
constant mean curvature (CMC-graph) equation for surfaces. In this work we use
this similarity to find a constraint on the Forchheimer flows, such that the
pressure function can be regarded as a surface with given constant mean curvature,
after a certain geometric transformation. Conversely, a
graph with prescribed constant mean curvature can be interpreted as a pressure
distribution of the flow subjected to a nonlinear Darcy law, similar to the
Forchheimer equation. This geometric interpretation provided us a with a simple
algorithm to compute the productivity index of the well, in a
structured heterogeneous reservoir. The productivity index (\textit{PI}) of the well is
one of the fundamental concepts in reservoir engineering, defined as
the ratio between the production rate and the difference between the well
pressure and the average pressure in the reservoir. The \textit{PI} characterizes the well
performance with respect to the geometry of the hydraulic system. It is shown that, {\it for some specific condition
that is reasonable to either impose or approximate}, the productivity index of the well in
the inhomogeneous reservoir can be computed using a solution of a corresponding CMC graph
equation.

\section{Introduction.}

\subsection{General Forchheimer equations}\label{gForch}
Darcy's law (for viscous fluid laminar flows) assumes that the total discharge is equal to the product between the medium permeability, the cross-sectional area of the flow, and the pressure drop, divided by the dynamic viscosity.
After dividing both sides of this equation by the area, one obtains another way to express Darcy's law: namely that the filtration velocity (or Darcy flux) is proportional to the pressure gradient, through a certain permeability coefficient. Darcy's equation, the continuity equation and the equation of state serve as the framework to model processes in reservoirs \cite{Muskat,Dake}.
For a slightly compressible fluid, the original PDE system reduces to a scalar
linear second order parabolic equation for the pressure only. The pressure
function is a major feature of the oil or gas filtration in porous
media, which is bounded by the well surface and the exterior reservoir boundary.
Different boundary conditions on the well correspond to different regimes of
production, while the condition on the exterior boundary models flux or
absence of flux into the drainage area. All together, the linear parabolic
equation, boundary conditions and some assumptions or guesses about the
initial pressure distribution form an IBVP.

There are different approaches for modeling non-Darcy's phenomena \cite{ELLP, Forch, Whitaker, Payne, RajTao}. It can be derived from the more general Brinkman-Forchheimer's equation \cite{Payne,ChadamQin}, or from mixture theory assuming certain relations between velocity field and ``drag-like" forces due to fluid to solid friction in the porous media \cite{Raj}. It can be also derived using homogenization arguments \cite{Sanchez}, or assuming some functional relation and then match the experimental data. In the current report, we just postulate a general constitutive equation relating the velocity vector field and the pressure gradient.
We will introduce constraints on the momentum equation and on the fluid density.
This will allow the reduction of the original system to a scalar quasi-linear parabolic equation
for the pressure only.

Hereafter, we use the following notations and basic definitions:

\begin{itemize}
\item  $v(x,t)$ represents the velocity field; $x$ is the spatial variable in $\R^d, d=2$ or $3$; $t$ denotes time;
$p(x,t)$ is the pressure distribution; $y\in \R^d$ are the variable vectors related to $\nabla p$; $s$, $\xi$
represent scalar variables;
\item The notations $C,C_0,C_1,C_2,\ldots$ denote generic positive constants not depending
on the solutions.
\end{itemize}

Current studies of flows in porous media widely use the three specific Forchheimer laws which we already mentioned (two-term, power, and three-term, respectively).
Darcy and Forchheimer laws can be written in vector forms as follows:
\begin{itemize}
\item {\bf Darcy's law:}
\beq\label{Darcy}
 \alpha v=-\nabla p,
\eeq
where $\displaystyle \alpha =\frac{\mu }{k}$ with $k$, in general, represents the permeability non-homogeneous function depending on $x$ subjected to the condition: $k_2^{-1} \ge\ k \ge\ k_2$,  $1 \ge k_2 > 0$. Here, the constant $\mu$ is the viscosity of the fluid.

\item {\bf The Forchheimer two-term law:}
\beq\label{F1}
 \alpha v+\beta |v|v=-\nabla p,
\eeq
where $\displaystyle \beta =\frac{\rho F\Phi}{k^{1/2}}$, $F$ is Forchheimer's coefficient, $\Phi$ is the porosity, and $\rho$ is the density of the fluid.

\item {\bf The Forchheimer power law:}
\beq\label{F2}
a v + c^n{|v|^{n-1}} v=-\nabla p,
\eeq
where $n$ is a real number belonging to the interval $[1,2]$. The strictly positive and bounded functions $c$ and $a$ are  found empirically, or can be taken as $c=(n-1)\sqrt{\beta}$ and $a=\alpha$.
By this way, $n=1$ and $n=2$ reduce the power law (\ref{F2}) to Darcy's law and to the Forchheimer two-term law, respectively.

\item {\bf The Forchheimer three-term law:}
 \beq\label{F3}
\mathcal{A} {v}+\mathcal{B}\, |v|v+\mathcal{C} |v|^2v=- \nabla p.
\eeq
Here $\mathcal{A}, \mathcal{B},$ and $\mathcal{C}$ are empirical constants.
\end{itemize}

We now introduce a {\it general form for the Forchheimer equations}.

\begin{definition}[$g$-Forchheimer Equations]\label{defgForch}
\beq\label{dafo-g} g(x,|v|)\, v=-\nabla p,\eeq
here $g(x,s)>0$ for all $s\ge 0$.
We will refer to \eqref{dafo-g} as $g$-Forchheimer (momentum) equations.
\end{definition}
Under isothermal condition the state equation relates the density $\rho$ to the pressure $p$ only, i.e.~ $\rho=\rho(p)$.
Therefore the equation of state has the form:
\beq \frac{1}{\rho}\frac{d \rho}{d p}=\frac1\kappa, \label{eqs} \eeq
where $1/\kappa$ is the compressibility of the fluid.
For slightly compressible fluid, such as compressible liquid, the compressibility is independent of pressure and is very small, hence  we obtain
\beq \rho=\rho_0 \exp\left(\frac{p-p_0}{\kappa}\right),\eeq
where $\rho_0$ is the density at the reference pressure $p_0$ (see \cite{Bear} Sec.~2.3, and also \cite{Muskat} Sec.~3.4).
Substituting Eq. \eqref{eqs} in the continuity equation
\beq \frac{d\rho}{dt}=-\nabla\cdot(\rho v),\eeq
yields
$$\frac{d\rho}{dp}\frac{dp}{dt}=-\rho\nabla\cdot v - \frac{d\rho}{dp} v\cdot\nabla p,$$
\beq\label{dafo-nonlin} \frac{dp}{dt}=-\kappa\nabla\cdot v - v\cdot\nabla p.\eeq
Since for most slightly compressible fluids in porous media the value of the constant $\kappa$ is large, following engineering tradition we drop the last term in \eqref{dafo-nonlin} and study the reduced equation:
\beq\label{lin-p} \frac{dp}{dt}=-\kappa \nabla\cdot v\,.\eeq

\subsection{Boundary conditions}\label{IBVPsection}

Let $U\subset \R^d$ be a $C^1$ domain modeling the drainage area in the porous media (reservoir), bounded by two boundaries: the exterior boundary $\Gamma_e$, and the accessible boundary $\Gamma_i$.

The exterior boundary $\Gamma_e$ models the geometrical limit of the well impact on the flow filtration and is often considered impermeable.
This yields the boundary condition:
\beq v\cdot N|_{\Gamma_e}=0,\eeq
where $N$ is the outward normal vector on the boundary $\Gamma=\Gamma_i\cup \Gamma_e$.
Other types of boundary conditions on the exterior boundary are discussed in \cite{SIAM}.

The accessible boundary $\Gamma_i$ models the well and defines the regime of filtration inside the domain.
On $\Gamma_i$, consider a given rate of production $Q(t)$, or a given pressure value $p=\varphi(x,t)$, or a combination of both.
It is very important from a practical point of view to build some ``baseline'' solutions capturing significant features of the well capacity and analyze the impact of the boundary conditions on these solutions. This analysis will be used to forecast the well performances and tune the model to the actual data.

On the boundary $\Gamma_i$, a ``split'' condition  of the following type is of particular interest:
\beq\label{S-BC} p=\psi(x,t)=\gamma(t)+\varphi(x),\eeq
where the time and space dependence of $p$ are separated.
This type of condition models wells which have conductivity much higher than the conductivity inside the reservoir.
The limiting homogeneous case $\psi(x)=const$ corresponds to the case of infinite conductivity on the well.

In case the flow is controlled by a given production rate $Q(t)$, the  solution is not unique.

Two important cases are:

(a) pressure distribution of the form $-At+\varphi(x)$, and

(b) constant total flux $ Q=const.$

The particular solutions of IBVP with boundary conditions (a) and (b) are ``time-invariant'' (see Section \ref{PSS}) and are used actively by engineers in their practical work.

\section{Non-Linear Darcy Equation and Monotonicity properties}
\label{Monotonicity}

In order to simplify the notation, we will further omit the $x$-dependence of $g$ in (\ref{dafo-g}).
Thus, one has
\beq\label{fluid-assumption} g(|v|)=g(x,|v|).\eeq
From \eqref{dafo-g} one has
\beq\label{G-eqn} g(|v|)|v|=-|\nabla p|,
\quad\hbox{for}\quad s\ge 0.\eeq

To make sure one can solve \eqref{G-eqn} for $|v|$, we impose the following conditions:
function $g$ belongs to $C([0,\infty))$ and $C^1((0,\infty))$, and satisfies	
\beq\label{G-cond} g(0)>0,
  \quad\hbox{and}\quad g'(s)\ge 0 \hbox{ for all } s> 0.
\eeq

Under this condition one has $\big(s g(s)\big)'=sg'(s) + g(s)\ge g(0)>0$, for any positive value $s$.
Therefore function $sg(s)$ is monotone and one can find $|v|$ as a function of $|\nabla p|$
\beq\label{absu} |v|=G(|\nabla p|)\eeq

Substituting equation \eqref{absu} into \eqref{dafo-g} one obtains the following alternative form of the
$g$-Forchheimer momentum equation \eqref{dafo-g}:
 \begin{definition}{(Non-linear Darcy Equation)} \cite{MF}
\beq\label{u-forma} v=\frac{-\nabla p}{g(G(|\nabla p|))} = -K(|\nabla p|)\nabla p,\eeq
where the function $K:[0,\infty)\to[0,\infty)$ is defined by
\beq\label{Kdef} K(\xi)=K_g(\xi)=\frac1{g(G(\xi))},\quad \xi\ge 0.\eeq
\end{definition}

Substituting \eqref{u-forma} for $u$ into \eqref{lin-p} one derives the degenerate parabolic equation
for the pressure:
\beq\label{lin-p-1} \frac{dp}{dt}=\nabla \cdot (K(|\nabla p|)\nabla p)\,.\eeq
where the compressibility constant $\kappa$ has been included in the non linear coefficient $K(|\nabla p|)$.

Some properties of the function $K$ can be found in \cite{MF}.
It turns our that the function $y\to K(|y|)y$ associated with the non-linear potential field on  the  RHS of equation \eqref{u-forma} is monotonic. This monotonicity and related properties are crucial in the study of the uniqueness and qualitative behavior of the the solutions of initial value problems (see e.g.~\cite{Evans}).

We illustrate the function $K$ for the particular case of two-term Forchheimer's equation.
This is one of the few cases when the function $K$ can be found explicitly.
\begin{example} \label{2termForchheimer}
 For the Forchheimer two-term law \eqref{F1}, let $g(s)=\alpha+\beta s$, then one has
$s g(s)=\beta  s^2 +\alpha s$ and
$s=G(\xi)=\frac{-\alpha+\sqrt{\alpha^2+4\beta \xi}}{2\beta}.$
Thus
\beqs K(\xi)=\frac1{\alpha+\beta G(\xi)}=\frac{2}{\alpha+\sqrt{\alpha^2+4\beta \xi}}.\eeqs
\end{example}

We now introduce the notion of generalized polynomial with positive coefficients and positive exponents, abbreviated as GPPC,
as it is a useful tool in this study.

\begin{definition}\label{GPPC}
We say that a function $g(s)$ is a GPPC if
\beq\label{gdef} g(s)=a_0 s^{\alpha_0} + a_1s^{\alpha_1}+a_2s^{\alpha_2}+\ldots +a_ks^{\alpha_k}=\sum_{j=0}^k a_j s^{\alpha_j},\eeq
where $k\ge 0$ represents a natural number, $0=\alpha_0<\alpha_1<\alpha_2<\ldots<\alpha_k$ represent real values, and the coefficients $a_0,a_1,\ldots, a_k$ are real and positive.
The largest exponent $\alpha_k$ is the degree of $g$ and is denoted by $\deg(g)$.

Class (GPPC) is defined as the collection of all GPPC.
\end{definition}

If the function $g$ in Definition~\ref{defgForch}  belongs to class (GPPC) then we call it the $g$-Forchheimer polynomial.

\begin{lemma}\label{K-estimate}
Let $g(s)$ be a function of class (GPPC)  as in \eqref{gdef}. Then $K(\xi)=K_g(\xi)$ is well-defined, is decreasing and satisfies
\beq\label{Kbound}   \frac{C_0}{1+\xi^{a}}\le K(\xi) \le \frac{C_1}{1+\xi^{a}} ,\;\forall \xi\ge 0,\eeq
where $a=\alpha_k/(\alpha_k+1)\in [0,1)$, and $C_0$ and $C_1$ are positive numbers depending on $a_j$'s and $\alpha_j$'s.
\end{lemma}
\noindent The proof of the previous Lemma can be found in \cite{MF} and makes use of a condition that is automatically satisfied by a GPPC.

Note that $a=0$ corresponds to the linear Darcy's case, while in the limiting case $a\rightarrow 1$ the largest exponent $\alpha_k$ diverges.
\section{Pseudo Steady State Solutions and Productivity Index}
\label{PSS}
In engineering and physics, it is often essential to identify special time-dependent pressure distributions that generate flows which are time-invariant.
In this section, we introduce the class of so-called pseudo-steady state (PSS) solutions which is used frequently  by reservoir and hydraulic engineers to evaluate the ``capacity" of the well (see. \cite{SIAM, AIVW3, Ibragimov1} and references therein).

\begin{definition}
A solution $\overline{p}(x,t)$ of the equation \eqref{lin-p-1} in domain $U$, satisfying the Neumann condition on $\Gamma_e$  is called a pseudo steady state (PSS) with respect to a constant $A$ if
\beq \frac{\partial \overline{p}(x,t)}{\partial t}=const.=-A\quad \hbox{for all}\quad t.\eeq
\end{definition}

Equation \eqref{lin-p-1} then implies

\beq\label{PSSeq} \frac{\partial \overline{p}(x,t)}{\partial t}= -A = \nabla \cdot (K(|\nabla \overline {p}|)\nabla \overline {p}).\eeq

Using Green's formula and the Neumann boundary condition on $\Gamma_e$ one derives
$$A|U|=-\int_{\Gamma_i} (K(|\nabla p|)\nabla p)\cdot N d\sigma = \int_{\Gamma_i} v\cdot N d\sigma =Q(t).$$
Therefore, the total flux of a PSS solution is time-independent
\beq\label{Poisson} Q(t)=A|U|=Q=const., \quad \hbox{for all}\quad t.\eeq

The PSS solutions  inherit two important features, which we will explore further.
On one hand, the total flux is defined by stationary equation \eqref{PSSeq} and is given.
On the other hand,  the trace of the solution on the boundary is split  \textit{a priori}.
Namely re-writing the PSS solution as
\beq\label{ph} \overline{p}(x,t)=-At + B + u(x),\eeq
one has $\nabla p=\nabla u$, hence $u$ and $p$ satisfy the same boundary condition on $\Gamma_e$. On $\Gamma_i$, in general, we consider
\beq\label{spss-phi} \overline{p}(x,t)=-At +B+\varphi(x), \quad\hbox{on} \quad \Gamma_i,\eeq
where $\varphi(x)$ is given and the constant $B$ is selected such that
\beq\int_{\Gamma_i}\varphi d\sigma=0.\eeq
Therefore $u(x)$ satisfies
\beq \label{h-eq}-A = \nabla \cdot K(|\nabla u|)\nabla u,\eeq
\beq\label{neumann-h} \frac{\partial u}{\partial N}=0 \quad\hbox{on}\quad \Gamma_e,\eeq
\beq\label{pss-h} u=\varphi \quad\hbox{on}\quad \Gamma_i.\eeq

Of particular interest is the case $\varphi(x)=0$ on $\Gamma_i$. From a physical point of view, it relates to the constraint that conductivity inside well is non-comparably higher than in the porous media.

We call $u(x)$ the profile of PSS corresponding to $A$ and the boundary profile $\varphi(x)$.

\begin{remark}
Note that for a PSS  as in \eqref{ph}, we have the quantity
\beqs
J(t)=\frac{Q(t)}{\frac1{|U|}\int_U {p}(x,t)dx-\frac1{|\Gamma_i|}\int_{\Gamma_i}{p}(x,t)d\sigma}=
\frac{Q}{\frac1{|U|}\int_U u(x)dx}.\eeqs
It represents the production rate versus the pressure drawdown (the difference between averages
in the domain and on the boundary
$\Gamma_i$), and is independent of time.
This quantity is called Productivity Index, and it is widely used by engineers
to test the performances of a well/reservoir system (see  \cite{SIAM, AIVW3,Ibragimov1}).
\end{remark}



\section{Geometric Interpretation }
In the PDE literature equation (see Serrin G., 1967, Gilbarg, D. {\&}
Trudinger, N.1977, Evans L. 1999)
\begin{equation}
\label{eq31}
div\left( {\left( {1+\left| {\nabla u} \right|^2} \right)^{-1/2}\nabla u}
\right)=-2H
\end{equation}
is referred as homogeneous CMC equation were $u$ is a graph defined in the
domain $U\subset R^2$.
We will also refer to it as a CMC graph equation. We noted
that this equation looks somewhat similar to
\begin{equation}
\label{eq32}
div\left( K(|\nabla u|)\nabla u \right)=-A,
\end{equation}
which has been previously introduced in defining the basic PSS profile for the non-linear Forchheimer equation.

It is not difficult to see that the non linear term in
\eqref{eq31} is ``about'' $(1+|\nabla u|)^{-1}$, which means that there exist constants $C_0$ and $C_1$
such that
\beq
\frac{C_0}{1+|\nabla u|}\le  \frac{1} {\left({1+\left| {\nabla u} \right|^2} \right)^{1/2}}\le\frac{C_1}{1+|\nabla u|}.
\eeq
We also showed in Lemma \ref{K-estimate} that the non linear term $K(|\nabla p|)$ in \eqref{eq32}
is, in the GPPC case, ``about'' $(1+|\nabla u|^a)^{-1}$. Here $a=\alpha_k/(\alpha_k+1)<1$, where $\alpha_k\ge0$ is the degree, $deg(g)$,
of the specific GPPC polynomial. Then in the limiting case $a \rightarrow 1$, we can expect the two non-linear coefficients in Eqs.
\eqref{eq31} and \eqref{eq32} to have the same structure.  Therefore one cannot expect the PSS Forchheimer equation to ``survive'' in limiting case $a=1$. At the same time it is worth mentioning that, for the case $a<1$, the weak solution of the PPS Forchheimer equation is unique (see \cite{MF}) and exists in the corresponding Sobolev space $W^{1,2-a}$. This result was treated in detail in \cite{MF}.

In the next section we will introduce a few basic
geometric notions and definitions, in order to show the robust link between
these two objects. The actual relationship between the two equations is far from
being straight-forward, in spite of their similarity. In particular, we will
show that the pressure function can be interpreted as a CMC graph under some
constraints and appropriate geometric transformations.

\section{The mean curvature equation for a graph.}

Any $C^2$ map $r:D\subset R^2\rightarrow R^3$ whose differential map $d r$
is injective is called an {\it immersion} (surface immersion) in $R^3$. Equivalently, the map $r$
represents an immersion if its Jacobian has rank 2, or all points
are {\it regular}. If an immersion $r$ is 1-1, it is sometimes
called an {\it embedding}. Any immersed surface can be endowed with
a general Riemannian metric $g=g(x,y)$ (\cite{Spi}, page 418). In
the particular case when the metric $g$ is defined using the usual
velocity vector fields, we will call it {\it naturally induced
metric} (i.e., naturally induced by the immersion).
In the most usual notation, $(M,g)$ represents a Riemannian
manifold of Riemannian metric $g$, and $M=r(D)$.

Consider a smooth surface that can represented as a graph
$z=u(x,y)$ of an open domain $D\subset R^2$. This surface is
parameterized via the map $r:D\subset R^2\rightarrow R^3$,
\begin{equation}
r(x,y)=(x,y,u(x,y))\,. \label{eq_r}
\end{equation}

\begin{definition} (Natural metric) We will call naturally induced metric the
following quadratic differential form:
\begin{equation}
dr^2(x,y)=|r_x|^2 dx^2+2\,<r_x,r_y>\, dx\,dy+|r_y|^2 dy^2\,,
\label{eq_rm}
\end{equation}
which can be rewritten as
\begin{equation}
dr^2(x,y)=(1+u_x^2)\,dx^2+2\,u_x\,u_y\;dx\,dy+(1+u_y^2)\,dy^2\,,
\label{eq_rm2}
\end{equation}
where the coefficients
$$g_{11}=1+u_x^2\,,\quad g_{12}=u_x\,u_y\,\quad g_{22}=1+u_y^2$$
represent the entries of the corresponding matrix $g$.
\end{definition}

\begin{definition} (Gauss Map) We will call Gauss map the usual unit normal
vector field $N:D\rightarrow S^2$ defined as:

\begin{equation}
N=\dfrac{r_x \times r_y}{\parallel r_x \times r_y \parallel}=
\dfrac{-u_x\,\mathbf{i}-u_y\,\mathbf{j}+\,\mathbf{k}}{\sqrt{u_x^2+u_y^2+1}}.\label{eq_N}
\end{equation}
\end{definition}

\begin{definition} (Second Fundamental form) The second fundamental form is defined by
\begin{equation}
d\sigma^2(x,y)=h_{11}\,dx^2+2\,h_{12}\,dx\,dy+h_{22}\,dy^2\,,
\label{eq_2ff}
\end{equation}
with the following coefficients
\begin{eqnarray}
&& h_{11}:=<N,r_{xx}> , \qquad h_{12}:=<N,r_{xy}>, \qquad h_{22}:=<N,r_{yy}>,
\end{eqnarray}
and so it can be expressed as
\begin{equation}
d\sigma^2(x,y)=\frac{u_{xx}}{\sqrt{u_x^2+u_y^2+1}}\,dx^2
         +2\frac{u_{xy}}{\sqrt{u_x^2+u_y^2+1}}\,dx\,dy
         +\frac{u_{yy}}{\sqrt{u_x^2+u_y^2+1}} \,dy^2\,.  \label{eq_2ff2}
\end{equation}
\end{definition}

Note that once the local coordinates are given, one can identify
the first and second fundamental forms with their corresponding
$2\times 2$ matrices, $g$ and $h$, respectively. The matrix
operator $S=g^{-1} h$ can be viewed as a linear operator from the
tangent plane to the surface at a given point, to the same tangent
plane. $S$ is usually called the Weingarten map or shape operator
(see \cite{Spi}, vol. II). Note that $g^{ij}\cdot h_{ij}$
represents its trace. The following result is a classical
result of differential geometry, which can be found in any text-book
(e.g.,  \cite{Spi}, vol. II), and whose proof is elementary.

\begin{proposition} For any first and second fundamental forms
defined for an immersion $r=r(x,y)$,  the following formula is
satisfied:
\begin{equation}
g^{ij}\cdot h_{ij}=2H,
\end{equation}
where $g$ and $h$ are matrices corresponding to the first and second
fundamental form, and $H$ represents the mean curvature, defined as the
arithmetic mean of the principal curvatures of the immersion $r$.
\end{proposition}
This equation is frequently referred to as {\bf mean curvature
equation}.

\begin{definition}(Laplace-Beltrami operator)
The Laplace-Beltrami operator $\Delta_g\,u$ corresponding to the
graph $z=u(x,y)$ and the metric defined in eq (\ref{eq_rm2}) is
defined as (see \cite{EL}):
\begin{equation}
\Delta_g\,u=\dfrac{1}{\sqrt{\det g}}\cdot g^{ij}
\dfrac{\partial^2u}{\partial x^i\,\partial x^j}.
\end{equation}
\end{definition}

 It is worthwhile noting that the Laplace-Beltrami operator is frequently defined
without the factor $\dfrac{1}{\sqrt{\det g}}$;  on the other hand,
this factor plays an important role in our work. Historically
speaking, its original definition is the same as ours. Also, it is
important to clarify that the Laplace Beltrami operator of a general
surface immersion $r$ is defined component-wise and represents a
vector $\Delta_g (r)$, which in our case it reduces to its last
component, $\Delta_g (u)$.

\cite{EL} was the first well-known reference to make the observation
that for surfaces immersed in $R^3$ the operator $g^{ij}\cdot
h_{ij}$ acting at each point coincides with the Laplace-Beltrami
operator. In view of this observation and the previous definition,
the mean curvature equation (13) can be rewritten as

$$\Delta_g\,u=2H.$$

A $C^2$ immersion $r:D\subset R^2\rightarrow R^3$ with vanishing
Laplace-Beltrami operator is said to be harmonic in a generalized
(Riemannian) sense.

Remark that for the case when the surface metric is the flat
Euclidean one $g$ represents the identity matrix and the
Laplace-Beltrami operator becomes the usual Laplace operator $\Delta
= \partial_{xx} + \partial_{yy}$.

A well-known result of geometric surface theory  stated
that an immersion $r$ as above is harmonic if and only if it is
minimal \cite {EL}. This result agrees with our previous Proposition and
Definition.

\ The mean curvature equation will be referred to as {\it constant
mean curvature equation} for the case when $H$ is constant. The
case of $H=const\ne0$ represents the case of CMC surfaces (to be
distinguished from the minimal case $H=0$).

We showed in \cite{AIT} that one can independently modify the
coordinate functions of the velocity vector, in a way that links the CMC equation
 to the Forchheimer-type equation.

 We started from the partial velocities $r_x$ and
$r_y$ at the point $P(x,y)$ of the initial graph
$r(x,y)=(x,y,u(x,y))$, and we applied the following transformation:
\begin{eqnarray}
&& {r_x}\rightarrow \tilde r_x:=(\chi,0, \mu(x,y) u_x(x,y)) \label{rxt}\\
&& {r_y}\rightarrow \tilde r_y:=(0,\chi, \mu(x,y) u_y(x,y))\,,\label{ryt}
\end{eqnarray}
 where
$\chi$ is a scaling constant and $\mu(x,y)$ is a smooth function.

Note that the notations $\tilde r_x$, $\tilde r_y$ do not a priori assume the existence of
an immersion $\tilde r$ whose partial velocity fields are written in this form.
We have studied the conditions in which there exists such a parametric surface.
This represents an easy case of the Frobenius theorem, and the existence condition of such
an immersion $\tilde r$ reduces to the compatibility condition $(\tilde r_x)_y=(\tilde r_y)_x$
being satisfied for the vectors defined above. For details, see Example 1.2.3, from \cite{Ivey},
based on successively applying Picard's Theorem in the $x$ and $y$-directions, respectively.
Assuming that this compatibility condition is satisfied, Example 1.2.3, from \cite{Ivey}, shows
that for any fixed initial position $P_0$ at the origin $(0,0)$ (or another fixed base point), there
exists a unique solution ${\tilde r}=(\tilde{x},\tilde{y},{\tilde u}(x,y))$, as an
immersion whose partial velocity vectors are ${\tilde r}_x$
and ${\tilde r}_y$.

The change $u(x,y)\rightarrow \tilde u(x,y)$ represents a smooth
deformation of the graph along the $z$-axis, while the change
$(x,y)\rightarrow (\tilde{x},\tilde{y})$, with $\tilde{x}=\chi\,x$ and $\tilde{y}=\chi\,y$
represents the rescaling of the graph domain from $D$ to $\tilde{D}$.
This non-trivial transformation modifies the Gauss map
(the tangent plane), as well as the shape operator.

\begin{proposition}
Consider a smooth graph in $R^3$ that is viewed as an immersion from
an open simply connected planar domain $D$ into the Euclidean space
$R^3$ via $r(x,y)=(x,y,u(x,y))$.

At every point on $M=r(D)$ consider the {\it modified velocity
vectors} defined as:
\

${\tilde r}_x=(\chi,0, \mu(x,y)u_x)$ and ${\tilde r}_y=(0,\chi,\mu(x,y)u_y)$,

where $\chi$ is constant and $\mu(x,y)$ is an arbitrary smooth function such
that ${\tilde r}_x$ and ${\tilde r}_y$ are non-vanishing and
linearly independent and such that the compatibility condition
$(\tilde r_x)_y=(\tilde r_y)_x$ is satisfied (or equivalently, $\mu_x u_y = \mu_y u_x$).

Consider a fixed initial position $P_0$ at the origin $(0,0)$ (or base point). Let
${\tilde r(x,y)}=(\tilde{x},\tilde{y}, \tilde u(x,y))$ represent the uniquely determined
integral surface having $\tilde r_x$ and $\tilde r_y$ as partial
velocity vectors, such that $\tilde r (0,0)=P_0$. These vectors naturally induce the first
fundamental form $\tilde g$. The corresponding acceleration vectors
together with the Gauss map ${\tilde N}(x,y)$ naturally induce the second fundamental form $\tilde h$.

Then the corresponding Laplace-Beltrami operator can be interpreted
in terms of the trace of the shape operator, that is,
\begin{equation}
\Delta_{\tilde g}\,{\tilde u}={\tilde g^{ij}}{\tilde h_{ij}}=2
\tilde H.
\end{equation}
\end{proposition}

\begin{remark} Note that $\tilde r$ actually represents a family
of immersions of parameter $\mu $, a real valued smooth function of
two variables. The initial immersion $r$ belongs to this family,
corresponding to the case $\chi=\mu=1$. We will call $z=\tilde{u}(x,y)$
{\it generalized graph}, and $\tilde g$ its naturally induced
metric.
\end{remark}

Example 1.2.3, from \cite{Ivey}, based on successively applying Picard's Theorem in the $x$ and $y$-directions, provides
an explicit solution for $\tilde r$ above.
Assuming that the compatibility condition is satisfied, we denote by
${\tilde r}=(\tilde{x},\tilde{y},{\tilde u}(x,y))$ an
immersion whose partial velocity vectors are ${\tilde r}_x$
and ${\tilde r}_y$. Some computational details can be found in \cite{AIT}. We therein collect the main information
on the first and second fundamental forms of the generalized graph, namely:

We first derive the expressions of the data corresponding to the immersion $\tilde r$. The induced metric is given by
\begin{equation}
d\tilde{r}^2(x,y)=(\chi^2+\mu^2 u_x^2)\,dx^2+2 \mu^2 u_x\,u_y\,dx\,dy+
(\chi^2+\mu^2 u_y^2)\,dy^2\,.
\end{equation}
The unitary normal vector field is given by
\begin{equation}
\tilde{N}=\dfrac{\tilde{r}_x\times\tilde{r}_y}{||\tilde{r}_x\times\tilde{r}_y||}=
\dfrac{-\mu\,u_x\,\mathbf{i}-\mu\,u_y\,\mathbf{j}+\chi\,\mathbf{k}}{\sqrt{\chi^2+\mu^2\,(u_x^2+u_y^2)}}\,.
\end{equation}
The coefficients of the second fundamental form, as entries of the corresponding matrix $\tilde{h}$, are:
\begin{eqnarray}
&& \tilde{h}_{11}=\dfrac{\chi\,(\mu\,u_{xx}+u_x\,\mu_x)}{\sqrt{\chi^2+\mu^2\,(u_x^2+u_y^2)}}, \\&&  \tilde{h}_{12}=
\dfrac{\chi\,(\mu\,u_{xy}+u_x\,\mu_y)}{\sqrt{\chi^2+\mu^2\,(u_x^2+u_y^2)}}=\dfrac{\chi\,(\mu\,u_{xy}+u_y\,\mu_x)}{\sqrt{\chi^2+\mu^2\,(u_x^2+u_y^2)}}\\
&& \tilde{h}_{22}=
\dfrac{\chi\,(\mu\,u_{yy}+u_y\,\mu_y)}{\sqrt{\chi^2+\mu^2\,(u_x^2+u_y^2)}}
\end{eqnarray}

The corresponding Laplace-Beltrami operator is given by
\begin{equation}
\Delta_{\tilde g}
=\dfrac{\mu\,(\chi^2+\mu^2\,u_y^2)\,u_{xx}-2\mu^3\,u_x\,u_y\,u_{xy}+
\mu\,(\chi^2+\mu^2\,u_x^2)\,u_{yy}}
{\chi\left(\chi^2+\mu^2\,(u_x^2+u_y^2)\right)^{3/2}}+
\dfrac{\chi({u_x\,\mu_x+u_y\,\mu_y})}
{\left(\chi^2+\mu^2\,(u_x^2+u_y^2)\right)^{3/2}}
\end{equation}

which can be rewritten as

\begin{eqnarray}
&&\Delta_{\tilde g}=
\nabla\cdot\left(
\dfrac{\mu \nabla u}
{\chi\sqrt{\chi^2+\mu^2|\nabla u|^2}}
\right)\,
=
\tilde{\nabla}\cdot\left(
\dfrac{\tilde{\nabla} \tilde{u}}
{\sqrt{1+|\tilde{\nabla} \tilde{u}|^2}}\right)
=2\,\tilde{H}\,.\label{CMCt}
\end{eqnarray}
Here $\tilde{\nabla} \cdot$ and  $\tilde{\nabla}$ are the divergence and the gradient
operator in the new reference system $(\tilde{x},\tilde{y})$.

\begin{remark}
In case $\tilde{H}$ is constant, $\tilde{u}$ is a CMC graph with respect to the domain $\tilde{D}$, provided a solution to (\ref{CMCt}) exists.
\end{remark}

The following theorem is an immediate consequence (see \cite{AIT} for details of the proof):

\begin{theorem}
Consider a smooth graph $u(x,y)$
in the domain $D(x,y)$,  solution of
\begin{eqnarray}
\nabla\cdot\big(K(|\nabla u|)\nabla u\big)=-A\,, \label{BVP}
\end{eqnarray}
and such that
\begin{equation}
(u_x\,u_{xy}+u_y\,u_{yy})\,u_x=(u_x\,u_{xx}+u_y\,u_{yx})\,u_y\,.\label{IL}
\end{equation}

Let $\tilde{D}(\tilde{x},\tilde{y})$ be the scaled domain obtained by the conformal mapping $\tilde{x}=\chi\,x$ and $\tilde{y}=\chi\,y$.
Let $\tilde{u}(\tilde{x},\tilde{y})$ be the stretched graph, which is parameterized as the immersion $\tilde r$
with partial velocities ($\tilde{r}_x$,$\tilde{r}_y$)  given by (\ref{rxt}-\ref{ryt})
and
\begin{eqnarray}
&& \mu(\chi,|\nabla u|)=\dfrac{-\chi K(|\nabla u|)}
{\sqrt{1-\chi^2 K(|\nabla u|)^2|\nabla u|^2}}\,, \label{mu}\\
&& \chi <\chi_{\max}=\dfrac{1}{|K(|\nabla u|)\,\nabla u|_{\max}}=\dfrac{1}{|v| _{\max}}\,. \label{chi}
\end{eqnarray}
Then $\tilde{u}(\tilde{x},\tilde{y})$ is solution of the corresponding CMC equation
\begin{equation}
\tilde{\nabla}\cdot\left(
\dfrac{\tilde{\nabla} \tilde{u}}
{\sqrt{1+|\tilde{\nabla} \tilde{u}|^2}}\right)
=A\,.\label{CMCta}
\end{equation}
\end{theorem}
\begin{remark}
Condition (\ref{chi}) assures $\mu$ to be a real smooth valued function.
\end{remark}

Condition (\ref{IL}) assures the compatibility condition
$$(\mu(\chi,|\nabla u|) u_x)_y=(\mu(\chi,|\nabla u|) u_y)_x$$
which are necessary and sufficient conditions for $\tilde u$ to exist.
Reformulated, (\ref{IL}) states that each level curve of the graph of $u$ (at $z=c_1$)
represents a level curve of the graph of $|\nabla u|$ (at $z=c_2$).

Replacing the generic function $\mu$ with (\ref{mu}) specified in the statement of this theorem into equation (\ref{CMCt}) immediately gives

\begin{equation}
\tilde{\nabla}\cdot\left(
\dfrac{\tilde{\nabla} \tilde{u}}
{\sqrt{1+|\tilde{\nabla} \tilde{u}|^2}}\right)=
-\nabla\cdot\big(K(|\nabla u|)\nabla u\big)\,=A
\end{equation}

This equality practically maps solutions of the Forchheimer equation (\ref{BVP}) to solutions of the CMC
equation (\ref{CMCt}) through an explicit transformation,
under a certain natural assumption.

\begin{proposition}
Consider a smooth graph $u(x,y)$ solution of the PSS Forchheimer equation (\ref{BVP})
and satisfying (\ref{IL}). Consider the  associated CMC graph $\tilde{u}(x,y)$ solution of (\ref{CMCta})
with $\mu$ given by (\ref{mu}) and $\chi<\chi_{max}$. Let $\eta=|\nabla u|$, $\xi=|\nabla \tilde{u}|$,
and $\tau=\frac{\xi}{\sqrt{1+\xi^2}}$, then
\beq \eta=g\big( |v(\tau,\chi)|\big)|v(\tau,\chi)|\,,\label{maineq}\eeq
where
\beq|v(\tau,\chi)|=\frac{\tau}{\chi}.\label{velocity}\eeq
Here $v$ is the velocity of the fluid flow in the porous media.
As a consequence
\begin{eqnarray}
&& u_x=-\eta{\frac{\tilde{u}_x }{\xi}}\,\label{ux}\\
&& u_y=-\eta{\frac{\tilde{u}_y }{\xi}}\,\label{uy}
\end{eqnarray}
\end{proposition}
\begin{proof}
The proof is based on the vector identity
$$\mu(|\nabla u|) \nabla u = \nabla \tilde u\,,$$
which implies the scalar equality
$$(\mu(|\nabla u|) |\nabla u|)^2 = |\nabla \tilde u|^2\,.$$

Taking into account equation (\ref{mu}) for $\mu$, and substituting
${\eta}$ for $|\nabla u|$ and $\xi$ for $|\nabla \tilde u|$,
we obtain
\beq\dfrac{(\chi K(\eta) \eta)^2} {1-(\chi K(\eta)\eta)^2}=\xi^2\,.\label{tau}\eeq

Let us consider $\tau=\chi K(\eta) \eta$, which is positive.

Substituting and solving for $\tau$ in \eqref{tau}, we obtain
\beq\tau(\xi)=\frac{\xi}{\sqrt{1+\xi^2}}.\eeq
Recollecting Eq. \eqref{u-forma} for the velocity field, it follows
\beq|v(\tau,\chi)|= K(\eta)\eta=\frac{\tau}{\chi}.\eeq	
Now by using the g-Forchheimer Eq. \eqref{gForch}, it follows Eq \eqref{maineq}.
Clearly, the correspondence between $\xi$ and $\eta$
is one-to-one.
\end{proof}

\subsection{Application to reservoir engineering}

The following proposition was proved in \cite{SIAM}. We recall it here
only because the proposed solution strategy for the evaluation of the productivity index is
among the direct applications of the results hereby presented.
\begin{proposition}
For the GPPC case the PSS Productivity Index can be computed as
\begin{equation}\label{alternative}
 PI (a_i,\alpha_i,|v|)= \frac{Q^2}{\int\limits_Ug(|v|)|v^2dx}=\frac{Q^2}{\int\limits_U\sum_{j=0}^k a_j|v|^{\alpha_j+2}dx}.
\end{equation}
Here $PI$ is time invariant, and it depends explicitely on $a_i, \alpha_i$ , $|v|$ and domain $U$.
\end{proposition}

This result combined with previous theorem (in particularly with Eq. \eqref{velocity} for the velocity)
expresses the fact that the Productivity Index of the well can be evaluated in the following way, provided that condition \eqref{IL} is verified.\\

\textbf{\textit{Solution strategy for the evaluation of the Productivity Index:}}
\begin{enumerate}
\item The factor $\chi$ is selected to generate the scaled domain $\tilde D$.
\item The CMC equation (\ref{CMCta}) is solved for $\tilde{u}(x,y)$
on the domain $\tilde{D}$ for $Q$, with the appropriate boundary conditions.
\item The gradient norm $\xi=|\nabla \tilde{u}|$ is evaluated.
\item The coefficients $a_i, \alpha_i$, $i=0,\cdots,k$ of the GPPC polynomial are selected.
\item The norm of the velocity $|v|$ is explicitely evaluated trough Eq. \eqref{velocity}.
\item The Productivity Index $PI(a_i,\,\alpha_i,\,|v|)$ is evaluated
by using formula (\ref{alternative}).
\end{enumerate}
\begin{remark}
Points 1 and 2  do not depend by the choice of $a_i$ and $\alpha_i$
which means that the evaluation of $\xi$ in 3 is $a_i$ and $\alpha_i$
independent. We need to solve just one BVP. The Productivity
Index can be evaluated a posteriori for any choice $a_i$ and $\alpha_i$.
\end{remark}

\begin{section} {Acknowledgments}
The authors would like to thank Luan Hoang for his stimulating discussions.
This research is supported by NSF Grant No. DMS-0908177.
\end{section}


\begin{thebibliography}{11}

\bibitem{MF} E. Aulisa, L. Bloshanskaya, L. Hoang, and A.Ibragimov,
\textit{Analysis of generalized Forchheimer flows of compressible fluids in porous media.} J. Math. Phys. Vol. 50, Available online (2009).

\bibitem{AIT} E. Aulisa, A. Ibragimov, M. Toda, Geometric Frame-Work for Modeling Non-Linear Flows in Porous Media and Its Applications in Engineering, Journal of Non-linear Analysis - Real World Applications, vol. 10 (2009); currently in press, avail. at Science Direct, doi:10.1016/j.nonrwa.2009.03.028

\bibitem{AIVW3} E. Aulisa, A. I. Ibragimov, P. P. Valk\'{o}, J. R. Walton,
\textit{Mathematical Frame-Work For Productivity Index of The Well for Fast Forchheimer (non-Darcy) Flow in Porous Media.} Mathematical Models and Methods in Applied Sciences, DOI:10.1142/S0218202509003772, in press.

\bibitem{SIAM} E. Aulisa, A. I. Ibragimov, P. P. Valk\'{o}, J. R. Walton,
\textit{A new method for evaluation the productivity index for non-linear flow},  SPE Journal,  DOI 10.2118/108984-PA, in press (2009).


\bibitem{Bear} J. Bear, \textit{Dynamics of Fluids in Porous Media,} Dover Publications, Inc., New York, 1972.

\bibitem{Lidia} L. Bloshanskaya, E. Aulisa and A. Ibragimov, \textit{Mathematical model of well productivity index for generalized Forchheimer flows and application.} Proceedings of SIAM Conference Mathematics for Industry: Challenges and Frontiers (MI09).


\bibitem{ChadamQin}{J. Chadam, Y. Qin},
\textit{Spatial decay estimates for flow in a porous medium},
SIAM J. Math. Anal., Vol. 28, No. 4, 808--830 (1997).

\bibitem{Dake} L. P. Dake,
\textit{Fundamental in reservoir engineering},
Elsevier, Amsterdam, (1978).


\bibitem{EL} J. Eells, L. Lemaire, {\it Two reports on harmonic maps}, World Scientific Pub Co Inc, July
1995, ISBN-10: 9810214669. First report.

\bibitem{ET} J. Eells, A. Ratto, {\it Harmonic maps and minimal immersions with symmetries},
Methods of Ordinary Differential Equations Applied to Elliptic
Variational Problems, Annals of Mathematics - Studies - Princeton
University Press, Study 130, 1993.


\bibitem{ELLP} E. Ewing,  R. Lazarov, S. Lyons, D. Papavassiliou, \textit{Numerical well model for non Darcy flow},
Comp. Geosciences, 3, 3-4, 185--204  (1999).


\bibitem{Evans} L. C. Evans, \textit{Partial Differential Equations.} American Mathematical Society, Providence, (1998).


\bibitem{Forch} P. Forchheimer,
\textit{Wasserbewegung durch Boden Zeit}, Ver. Deut. Ing. 45, (1901).


\bibitem{Muskat} M. Muskat,
\textit{The flow of homogeneous fluids through  porous media},
McGraw-Hill Book Company, Inc., New York and London, (1937).


\bibitem{Ibragimov1} A. I. Ibragimov, D. Khalmanova, P. P. Valk\'{o}, J. R. Walton,
\textit{On a mathematical model of the productivity index of a well from reservoir engineering,}
SIAM J. Appl. Math., 65, 1952--1980  (2005).


\bibitem{Ivey} T. A. Ivey, J. M. Landsberg, Cartan for Beginners, Differential Geometry of Moving Frames and Exterior Differential Systems,
Graduate Studies in Mathematics, AMS, vol. 61, 2003.

\bibitem{Miskimins}
J.L. Miskimins, H.D. Lopez-Hernandez,  R.D. Barree,  \textit{Non-Darcy Flow in Hydraulic Fractures}, JPT, March, pp. 57-59, 2006

\bibitem{Payne} L. E. Payne, B. Straughan,
\textit{Convergence and Continuous Dependence for the Brinkman-Forchheimer Equations},
Studies in Applied Mathematics, 102, 419--439 (1999).



\bibitem{Raj} K. R. Rajagopal, \textit{On a hierarchy of approximate models for flows of incompressible fluids through porous solids},
Mathematical Models and Methods in Applied Sciences, Vol. 17, No. 2, 215--252  (2007).

\bibitem{RajTao} K. R. Rajagopal, L. Tao, \textit{Mechanics of mixtures},
World scientific, (1995).



\bibitem{Sanchez} E. Sanchez-Palencia,
\textit{Non-Homogeneous Media and Vibration Theory}, Lecture Notes in Physics, Springer-Verlag, (1980).


\bibitem{Spi} M. Spivak, {\it Differential geometry}, vol. I and II, 2nd
Edition, Publish or Perish, 1977.

\bibitem{Tavera}
C.A.P. Tavera, H. Kazemi, E. Ozkan, \textit{Combine effect of Non-Darcy Flow and Formation Damage on Gas Well Performance of Dual-Porosity and Dual Permeability Reservoirs}, SPE-90623, 2004.

\bibitem{Whitaker} S. Whitaker, \textit{The Forchheimer Equation: A Theoretical Development,}
Transport in Porous Media 25: 27--61 (1996).

\end{thebibliography}
\end{document}